\documentclass[10pt]{amsart}
\usepackage{graphicx}
\usepackage{amscd}
\usepackage{amsmath}
\usepackage{amsthm}
\usepackage{amsfonts}
\usepackage{amssymb}
\usepackage{mathrsfs}
\usepackage{bm}
\usepackage{enumerate}
\usepackage{amsrefs}
\usepackage{xcolor}
\usepackage[colorlinks, citecolor=blue, linkcolor=red, pdfstartview=FitB]{hyperref}

\usepackage{marginnote}	
\usepackage{soul} 			

\newcommand{\bC}{{\mathbb{C}}}

\newcommand{\bM}{{\mathbb{M}}}

  \newcommand{\A}{{\mathcal{A}}}

  \newcommand{\D}{{\mathcal{D}}}
  \newcommand{\E}{{\mathcal{E}}}

  \newcommand{\R}{{\mathcal{R}}}

  \newcommand{\U}{{\mathcal{U}}}

\newcommand{\fd}{{\mathfrak{d}}}

\newcommand{\fs}{{\mathfrak{s}}}

\newcommand{\rC}{\mathrm{C}}

\newcommand{\eps}{\varepsilon}
\renewcommand{\phi}{\varphi}

\newcommand{\upchi}{{\raise.35ex\hbox{$\chi$}}}

\newcommand{\ol}{\overline}

\newcommand{\qand}{\quad\text{and}\quad}

\newcommand{\Cone}{\operatorname{Cone}}
\newcommand{\Order}{\operatorname{Order}}

\newcommand{\id}{\operatorname{id}}

\newcommand{\re}{\operatorname{Re}}

\newcommand{\tr}{\operatorname{tr}}

\newtheorem{lemma}{Lemma}[section]
\newtheorem{theorem}[lemma]{Theorem}
\newtheorem{proposition}[lemma]{Proposition}
\newtheorem{corollary}[lemma]{Corollary}
\newtheorem{theoremx}{Theorem}

\theoremstyle{definition}

\date{\today}
\author{Rapha\"el Clou\^atre}
\author{Hridoyananda Saikia}
\address{Department of Mathematics, University of Manitoba, Winnipeg, Manitoba, Canada R3T 2N2}

\email{raphael.clouatre@umanitoba.ca}
\email{saikiah@myumanitoba.ca}
\thanks{R.C. was partially supported by an NSERC Discovery Grant.}

\title{A boundary projection for the dilation order}

\begin{document}
\begin{abstract}
Motivated by Arveson's conjecture, we introduce a notion of hyperrigidity for a partial order on the state space of a $\rC^*$-algebra $B$. We show how this property is equivalent to the existence of a \emph{boundary}: a subset of the pure states which completely encodes maximality in the given order.
In the classical case where $B$ is commutative, such boundaries are known to exist when the partial order is induced by some well-behaved cone. However, the relevant order for the purposes of Arveson's conjecture is the dilation order, which is not known to fit into this framework. Our main result addresses this difficulty by showing that the dilation maximal states are stable under absolute continuity. Consequently, we obtain the existence of a boundary projection in the bidual $B^{**}$, on which all dilation maximal states must be concentrated. The topological regularity of this boundary projection is shown to lie at the heart of Arveson's conjecture. Our techniques do not require $B$ to be commutative.

\end{abstract}
\maketitle

\section{Introduction}
Let $B$ be a unital $\rC^*$-algebra generated by an operator system $S\subset B$. A unital $*$-representation $\pi:B\to B(H)$ is said to have the \emph{unique extension property with respect to $S$} if any unital completely positive map $\psi:B\to B(H)$ agreeing with $\pi$ on $S$ must in fact agree with $\pi$ everywhere on $B$. This special class of $*$-representations, and especially its irreducible elements called \emph{boundary representations}, first rose to prominence as the centrepiece of Arveson's program for constructing the $\rC^*$-envelope of $S$ \cite{arveson1969},\cite{arveson1972}, which eventually came to fruition in \cite{dritschel2005},\cite{arveson2008},\cite{davidsonkennedy2015}. 

Beyond this important application, $*$-representations with the unique extension property have since become a mainstay in modern non-selfadjoint operator algebra theory, where they can be meaningfully interpreted as forming the non-commutative Choquet boundary of $S$; see \cite{DK2019},\cite{CTh2023} and the references therein. 

In \cite{arveson2011}, Arveson showed that a certain rigidity property of $S$ inside of $B$, inspired by a classical phenomenon observed by Korovkin in approximation theory, could be reformulated using the unique extension property. More precisely, we say that $S$ is \emph{hyperrigid} if every unital $*$-representation of $B$ has the unique extension property with respect to $S$. Guided by \v Sa\v skin's explanation of the Korovkin phenomenon \cite{saskin1967}, Arveson ventured the following.

\vspace{3mm}

\textbf{Hyperridigity conjecture.} The operator system $S$ is hyperrigid in $B$ whenever all irreducible $*$-representations of $B$ are boundary representations for $S$.

\vspace{3mm}

In essence, Arveson conjectured that to verify that the unique extension property holds for all $*$-representations, it suffices to determine whether it holds for all \emph{irreducible} $*$-representations. This conjecture is still wide open, despite attracting attention for more than a decade \cite{kleski2014hyper},\cite{kennedyshalit2015},\cite{clouatre2018unp},\cite{clouatre2018lochyp},\cite{salomon2019},\cite{KR2020},\cite{kim2021}. Notably,  even the case where $B$ is commutative has not been settled; see \cite{DK2021} for some recent work in this direction. Our aim in this paper is to shed new light on the conjecture, with no requirement of commutativity.

Basic theory of $\rC^*$-algebras guarantees that any $*$-representation is a direct sum of cyclic ones. In addition, the unique extension property passes both to direct sums and to subrepresentations \cite[Lemma 2.8]{CTh2022}, so that $S$ is hyperrigid in $B$ whenever all \emph{cyclic} $*$-representations of $B$ have the unique extension property with respect to $S$. Recent work of Davidson and Kennedy \cite{DK2019},\cite{DK2021} exhibited  a mechanism for tackling this issue.

Let $\phi$ be a state on $B$. Let $\pi:B\to B(H)$ be a unital $*$-representation and let  $\xi\in H$ be a unit vector. We say that the triple $(\pi,H,\xi)$ is a \emph{representation of $\phi$} if
\[
\phi(b)=\langle \pi(b)\xi,\xi\rangle, \quad b\in B.
\]
Via the GNS construction, there is a one-to-one correspondence between states and cyclic $*$-representations. 
As shown in \cite{DK2019},\cite{DK2021}, the unique extension property for a cyclic $*$-representation $\pi$ is entirely encoded in the corresponding state $\phi$. More precisely, it follows from \cite[Theorem 8.3.7]{DK2019} that $\pi$ has the unique extension property with respect to $S$ precisely when the state $\phi$ is maximal in a certain partial order on the state space of $B$, called the \emph{dilation order} (the precise definition of this order will be given in Section \ref{S:dilordercone}). Consequently, a deeper  understanding of the dilation maximal states is highly relevant for the purposes of Arveson's conjecture, and this is what our work aims to provide.

In the special case where $B$ is commutative, the authors of \cite{DK2021} capitalize on this alternative characterization of the unique extension property by relating the dilation order to another order, much more classical and transparent, called the Choquet order. Crucially, the maximal elements in the Choquet order are completely understood as those states on $B$ concentrated on a certain subset, called the Choquet boundary. As explained in \cite[Question 8.4]{DK2021}, Arveson's conjecture, at least in the commutative setting, then becomes equivalent to determining whether the dilation and Choquet orders share the same maximal elements. We explore this theme further, and examine abstract partial orders on state spaces of arbitrary $\rC^*$-algebras. 
 
For  a unital $\rC^*$-algebra $B$, we let $\E(B)$ denote its state space, and $\E_p(B)$ denote the pure states. Given a state $\phi$ on $B$, we let $\R_\phi$ denote the set of Borel probability measures $\mu$ on $\E_p(B)$ satisfying $\phi=\int \omega d\mu(\omega)$. Such measures always exist, at least in the separable setting \cite[Theorem 4.2]{bishop1959}.

Given a Borel measurable subset $X\subset \E_p(B)$, we let $\Sigma_X$ denote the set of those states $\phi$ for which there exists $\mu\in \R_\phi$ concentrated on $X$. Furthermore, we let $\Sigma^X$ denote the set of those states $\phi$ for which every $\mu\in \R_\phi$ is concentrated on $X$.

Let $\Delta\subset \E(B)\times \E(B)$ be partial order, and denote its maximal elements by $\max(\Delta)$.
A Borel measurable subset $X\subset \E_p(B)$ will be called a $\Delta$-\emph{boundary} if 
\[
\max(\Delta)=\Sigma^X=\Sigma_X.
\]
We say that the order $\Delta$ is \emph{hyperrigid} if $\Sigma_\Omega\subset \max(\Delta)$, where $\Omega=\E_p(B)\cap \max(\Delta)$ is the set of pure $\Delta$-maximal states.  In other words, $\Delta$ is hyperrigid precisely when states of the form $\int \omega d\mu(\omega)$ are $\Delta$-maximal, where $\mu$ is a Borel probability measure concentrated on the pure $\Delta$-maximal states. 

In Section \ref{S:boundaries}, the relationships between the above notions are studied. In Corollary \ref{C:bdryHR}, we show that hyperrigidity of $\Delta$ is equivalent to the existence of a $\Delta$-boundary, at least when $B$ is separable and $\Delta$ is weak-$*$ closed and convex.

In Section \ref{S:dilordercone}, we define the dilation order $\D(S,B)\subset \E(B)\times \E(B)$ relative to an operator system $S$ generating $B$. Using the terminology introduced above, in Corollary \ref{C:Arvreform} we show that Arveson's conjecture can be reformulated as saying that  $\D(S,B)$ is hyperrigid whenever all pure states are $\D(S,B)$-maximal. In turn, this is equivalent to the existence of a $\D(S,B)$-boundary. 

Classically, it is known that certain partial orders on the state space of a commutative $\rC^*$-algebra always admit boundaries. For instance, in  \cite[Corollary I.5.18]{alfsen1971} the condition that the order be induced by some cone in $B$ that is stable under taking maxima is shown to be sufficient.  We thus examine the dilation order from this perspective, so as to determine whether the aforementioned machinery can be employed to produce the boundary that we seek. To do this, we rely rather heavily on the developments in non-commutative Choquet theory and function theory obtained recently by Davidson and Kennedy \cite{DK2019}.

When $B$ is chosen to be the so-called maximal $\rC^*$-cover of $S$, then we show in Proposition \ref{P:dilordercone} that the dilation order is indeed induced by some cone $\Xi$, whose closure is unique with this property (Corollary \ref{C:dilconeunique}). For a general representation of $S$, these facts still allow us to identify the dilation maximal states in terms of an order induced by a cone (Theorem \ref{T:dilmaxinv}).  Unfortunately, even when $S$ can be represented in a commutative $\rC^*$-algebra, the corresponding cone is \emph{not}  stable under maxima (Proposition \ref{P:maxstable}), so that the aforementioned result from \cite{alfsen1971} cannot be used to manufacture a boundary and thus resolve Arveson's conjecture. 

In Section \ref{S:bdryproj}, we will see how this difficulty in exhibiting a boundary can be somewhat circumvented, provided that one is willing to look for a ``non-classical" boundary. The key observation that we make is that  $\D(S,B)$-maximality is preserved under absolute continuity (Theorem  \ref{T:abscont}), thus enabling the use of some results from \cite{CTh2023} on non-commutative measure theory. The main result of the paper is then the following (see Theorem \ref{T:bdryproj}).

\begin{theoremx}\label{T:A}
Let $B$ be a unital $\rC^*$-algebra and let $S\subset B$ be an operator system such that $B=\rC^*(S)$. Then, there exists a projection $\fd \in B^{**}$ with the property that a state $\phi$ on $B$ is $\D(S,B)$-maximal precisely when $\phi(\fd)=1$.
\end{theoremx}

We refer to the projection $\fd$ above as the \emph{boundary projection}. We examine the non-commutative topological properties of $\fd$, in the sense of Akemann \cite{akemann1969},\cite{akemann1970left}. As an application, our second main result reformulates Arveson's conjecture in terms of regularity properties of $\fd$; see Corollary \ref{C:HRtop}.

\begin{theoremx}\label{T:B}
Assume that $B$ is separable and that every pure state on $B$ is $\D(S,B)$-maximal. 
Then, the following statements are equivalent.
\begin{enumerate}[{\rm (i)}]
\item The operator system $S$ is hyperrigid in $B$.
\item The boundary projection $\fd$ is closed.
\item The boundary projection $\fd$ is the infimum of a collection of open projections in $B^{**}$.
\end{enumerate}
\end{theoremx}

\section{Boundaries and hyperrigidity for pre-orders}\label{S:boundaries}
 
Let $B$ be a unital $\rC^*$-algebra, let $\E(B)$ denote its state space, and let $\E_p(B)$ denote the pure states. In this section,  we aim to understand the structure of the maximal elements in some pre-orders defined on $\E(B)$. In the next section, our findings will be applied to a specific example of a partial order, but we proceed here in greater generality.

We begin with a technical fact.

\begin{lemma}\label{L:Borel}
Let $B$ be a separable unital $\rC^*$-algebra, and let $\Delta\subset \E(B)\times \E(B)$ be a weak-$*$ closed pre-order.
Let $\{a_n:n\geq 1\}$ be a countable dense subset of the self-adjoint part of $B$. For each pair of integers $m,n\geq 1$, let $K_{n,m}\subset \E(B)$ consist of those states $\phi$ for which  there is another state $\psi$ such that $(\phi,\psi)\in \Delta$ and
$
\psi(a_n)-\phi(a_n)\geq 1/m.
$
Then, each $K_{n,m}$ is weak-$*$ closed, and  $\bigcup_{n,m=1}^\infty K_{n,m}$ is  the set of states on $B$ that are not $\Delta$-maximal. In particular, the set of $\Delta$-maximal states is Borel measurable.
\end{lemma}
\begin{proof}
It is clear that $\bigcup_{n,m=1}^\infty K_{n,m}$ is contained in the set of states on $B$ that are not $\Delta$-maximal. Conversely, if $\phi$ is a state on $B$ which is not $\Delta$-maximal, then there is another state $\psi\neq \phi$ such that $(\phi,\psi)\in \Delta$. This implies that there must be a self-adjoint element $b\in B$ such that $\phi(b)\neq \psi(b)$. Upon replacing $b$ with $-b$ if necessary, we may assume that $\psi(b)-\phi(b)>0$. The density of the set $\{a_n\}$ then easily implies that $\phi \in \bigcup_{n,m=1}^\infty K_{n,m}$. We thus conclude that $\bigcup_{n,m=1}^\infty K_{n,m}$ is the set of states on $B$ that are not $\Delta$-maximal.

Fix integers $m,n\geq 1$, and let $(\phi_i)$ be a net in $K_{n,m}$ converging to some state $\phi\in \E(B)$ in the weak-$*$ topology. By definition, this means that there is another net  of states $(\psi_i)$ such that $(\phi_i,\psi_i)\in \Delta$ and $\psi_i(a_n)-\phi_i(a_n)\geq 1/m$.  Upon passing to a cofinal subnet, we may assume that $(\psi_i)$ also converges to some state $\psi\in \E(B)$ in the weak-$*$ topology. Clearly, we then have $
\psi(a_n)-\phi(a_n)\geq 1/m,
$
while $(\phi,\psi)\in \Delta$ since $\Delta$ is assumed to be weak-$*$ closed. This shows that $\phi\in K_{n,m}$, so indeed $K_{n,m}$ is closed in the weak-$*$ topology.

Finally, the previous paragraph implies that the $\Delta$-maximal states form a $G_\delta$-set, and hence a Borel measurable set.
\end{proof}

Given a state $\phi$ on $B$, we let $\R_\phi$ denote the set of Borel probability measures $\mu$ on $\E(B)$ concentrated on $\E_p(B)$ and satisfying 
\[
\phi(b)=\int \omega(b) d\mu(\omega),\quad b\in B.
\]
When $B$ is separable, such measures always exist  \cite[Theorem 4.2]{bishop1959}.

The following is inspired by the proof of \cite[Corollary 3.3]{bishop1959}, and it generalizes the separable version of \cite[Proposition 9.2.5]{DK2019} to a large class of pre-orders.

\begin{theorem}\label{T:BdL}
Let $B$ be a separable unital $\rC^*$-algebra. Let $\Delta\subset \E(B)\times \E(B)$ be a weak-$*$ closed, convex pre-order. Let $\phi$ be a $\Delta$-maximal state on $B$ and let $\mu$ be a measure in $\R_\phi$. Then, $\mu$ is concentrated on the pure $\Delta$-maximal states.
\end{theorem}
\begin{proof}
First note that because $B$ is separable, the set $\E_p(B)$ is Borel measurable \cite[Corollary 3.3 and Lemma 4.1]{bishop1959}.  Let $N\subset \E_p(B)$ denote the set of pure states that are not $\Delta$-maximal. Then, $N$ is Borel measurable by Lemma \ref{L:Borel}. Our goal is to show that $\mu(N)=0$.

Assume for the sake of contradiction that $\mu(N)>0$. By Lemma \ref{L:Borel}, there is a self-adjoint element $a\in B$ and $\eps>0$ such that $\mu(K)>0$, where $K$ is the set of pure states $\omega$ on $B$ for which there is a state $\psi$ with $(\omega,\psi)\in \Delta$ and 
$
\omega(a)-\psi(a)\geq \eps.
$
Define a non-zero positive linear functional $\phi'$ on $B$ as
\[
\phi'(b)=\int_{K}\omega(b)d\mu(\omega), \quad b\in B.
\]
As is well known, we can find a net $(\alpha_i)$ of finitely supported Borel probability measures on $K$ such that 
\[
\lim_i \int_K fd\alpha_i=\frac{1}{\mu(K)}\int_K fd\mu, \quad f\in \rC(K).
\]
Correspondingly, for each $i$ we may define a state $\beta_i$ on $B$ as
\[
\beta_i(b)=\int_K \omega(b)d\alpha_i(\omega), \quad b\in B.
\]
Note then that $(\beta_i)$ converges to $\frac{1}{\mu(K)}\phi'$ in the weak-$*$ topology of $B^*$.
By definition of $K$, since each $\alpha_i$ is a finite convex combination of point masses on $K$ and $\Delta$ is convex, for each $i$ there is a state $\psi_i$ on $B$ such that $(\beta_i,\psi_i)\in \Delta$
while $\psi_i(a)-\beta_i(a)\geq \eps$. Let $\psi\in \E(B)$ be a weak-$*$ cluster point of $(\psi_i)$. Using that $\Delta$ is weak-$*$ closed, taking the weak-$*$ limit of a cofinal subnet of $(\beta_i,\psi_i)$, we find $(\frac{1}{\mu(K)}\phi',\psi)\in \Delta$ and 
\[
\psi(a)-\frac{1}{\mu(K)}\phi'(a)\geq \eps.
\]
In particular, $\psi\neq \frac{1}{\mu(K)}\phi'$.
Since $\phi$ is assumed to be $\Delta$-maximal, $\phi\neq \frac{1}{\mu(K)}\phi'$ and thus  $0<\mu(K)<1$.

Finally, set $\theta=(\phi-\phi')+\mu(K)\psi$. Then, $\theta$ is a convex combination of the state $\psi$ and the state $\phi''=\frac{1}{1-\mu(K)} (\phi-\phi')$, and hence is  state itself. We note that
\[
(\phi,\theta)=(1-\mu(K))(\phi'',\phi'')+\mu(K)\left(\frac{1}{\mu(K)}\phi',\psi\right)
\]
so that $(\phi,\theta)\in \Delta$ by convexity. Observe that
\[
\theta(a)-\phi(a)=\mu(K)\left (\psi(a)-\frac{1}{\mu(K)}\phi'(a)\right)\geq \mu(K)\eps
\]
so that $\phi\neq \theta$, which contradicts the fact that $\phi$ is $\Delta$-maximal. Consequently, $\mu(N)=0$ as desired.
\end{proof}

Next, we examine a converse to Theorem \ref{T:BdL}. 
For a Borel measurable subset $X\subset \E_p(B)$, we let $\Sigma_X$ denote the set of those states $\phi$ on $B$ for which there exists $\mu\in \R_\phi$ concentrated on $X$. Furthermore, we let $\Sigma^X$ denote the set of those states $\phi$ on $B$ for which every $\mu\in \R_\phi$ is concentrated on $X$. Plainly, we have $\Sigma^X\subset \Sigma_X$ as long as $\R_\phi$ is non-empty. It follows immediately from \cite[Lemma 4.1]{bishop1959} that if $\omega$ is a pure state on $B$, then $\R_\omega$ consists only of the point mass at $\omega$. Hence
\begin{equation}\label{Eq:Bauer}
\Sigma^X\cap \E_p(B)=\Sigma_X\cap \E_p(B)=X.
\end{equation}

Let $\Delta\subset \E(B)\times \E(B)$ be pre-order, and denote its maximal elements by $\max(\Delta)$.  Under natural conditions, we can show that maximal elements are always plentiful, at least for partial orders.

\begin{proposition}\label{P:maxexistence}
Let $B$ be a unital $\rC^*$-algebra. Let $\Delta\subset \E(B)\times \E(B)$ be a weak-$*$ closed partial order. For every state $\phi$ on $B$, there exists a $\Delta$-maximal state  $\theta$ such that $(\phi,\theta)\in \Delta$.
\end{proposition}
\begin{proof}
Let $Z$ denote the set of states $\psi$ on $B$ such that $(\phi,\psi)\in \Delta$. Let $C\subset Z$ be a chain. There is a cofinal subnet  $\Lambda$ of $C$ that converges to some $\tau$ in the weak-$*$ topology. Since $\Delta$ is assumed to be weak-$*$ closed, we see that $(\phi,\tau)\in \Delta$, whence $\tau\in Z$. 

Next, fix $\psi\in C$. Then, $\Lambda_\psi=\{\lambda\in \Lambda: (\psi,\lambda)\in \Delta\}$ is a cofinal subnet of $\Lambda$. Using once again that $\Delta$ is weak-$*$ closed, we find
\[
(\psi,\tau)=\lim_{\lambda\in \Lambda_\psi} (\psi,\lambda)\in \Delta.
\]
We conclude that $\tau$ is an upper bound for $C$. By Zorn's lemma, $Z$ has a maximal element $\theta$. 

We claim that $\theta$ is in fact $\Delta$-maximal. To see this, assume that $\gamma$ is a state on $B$ such that $(\theta,\gamma)\in \Delta$. Then, $(\phi,\gamma)\in \Delta$ so that $\gamma\in Z$. Maximality of $\theta$ in $ Z$ then forces $\theta=\gamma$, thereby completing the proof.
\end{proof}

A Borel measurable subset $X\subset \E_p(B)$ will be called a $\Delta$-\emph{boundary} if 
\[
\max(\Delta)=\Sigma^X=\Sigma_X.
\]
We also set $\Omega=\max(\Delta)\cap \E_p(B)$, that is, $\Omega$ is the set of pure $\Delta$-maximal states. Recall that by  \cite[Corollary 3.3 and Lemma 4.1]{bishop1959} and Lemma \ref{L:Borel}, $\Omega$ is Borel measurable whenever $B$ is separable. We say that the order $\Delta$ is \emph{hyperrigid} if $\Sigma_\Omega\subset \max(\Delta)$. In other words, this says that a state of the form $\int \omega d\mu(\omega)$ is $\Delta$-maximal if $\mu$ is a Borel probability measure concentrated on $\Omega$. This condition is vacuously satisfied if $\Delta$ has simply no maximal elements.

The following is the main result of this section, and it shows that hyperrigidity and boundaries are closely related.

\begin{corollary}\label{C:bdryHR}
Let $B$ be a separable unital $\rC^*$-algebra. Let $\Delta\subset \E(B)\times \E(B)$ be a weak-$*$ closed, convex pre-order. Then, the following statements are equivalent.
\begin{enumerate}[{\rm(i)}]
\item There is a Borel measurable subset $X\subset \E_p(B)$ such that $ \max(\Delta)= \Sigma_X$.
\item The set $\Omega$ of pure $\Delta$-maximal states is a $\Delta$-boundary.
\item The order $\Delta$ is hyperrigid.
\end{enumerate}
\end{corollary}
\begin{proof}
(i) $\Rightarrow$ (ii): Using the assumption along with \eqref{Eq:Bauer}, we find
\[
\Omega=\max(\Delta)\cap \E_p(B)=\Sigma_X\cap \E_p(B)=X.
\]
Therefore, $\max(\Delta)=\Sigma_\Omega$. By virtue of Theorem \ref{T:BdL}, we then find
\[
\Sigma^\Omega\subset \Sigma_\Omega=\max(\Delta)\subset \Sigma^\Omega
\]
which implies that $\Omega$ is a $\Delta$-boundary.

(ii) $\Rightarrow$ (i) + (iii):  This is trivial.

(iii) $\Rightarrow$ (ii): Assume that $\Delta$ is hyperrigid, that is, $\Sigma_\Omega\subset \max(\Delta)$. Invoking Theorem \ref{T:BdL}, we see that $\max(\Delta)\subset \Sigma^\Omega$, so that $\max(\Delta)=\Sigma^\Omega=\Sigma_\Omega$ and $\Omega$ is a $\Delta$-boundary. 
\end{proof}

\section{The dilation order}\label{S:dilordercone}

The main driving force of this paper is to ascertain the hyperrigidity (in the sense of Section \ref{S:boundaries}) of the so-called dilation order relative to an operator system, as introduced in \cite{DK2021},\cite{DK2019}. The goal of this section is to give a precise definition of this order, and to establish some of its properties. The results therein all rely heavily on the technical machinery of non-commutative function theory developed in \cite{DK2019}. In an effort to keep the exposition light, we only recall the bare minimum from that paper. The interested reader should consult \cite{DK2019} to fill in the gaps as needed.

\subsection{Definition}

Let $B$ be a unital $\rC^*$-algebra and let $S\subset B$ be an operator system such that $\rC^*(S)=B$. Let $\E(B)$ denote the state space of $B$. Let $\phi\in \E(B)$. Recall that by a \emph{representation} of $\phi$ we mean a triple $(\pi,H,\xi)$ consisting of a Hilbert space $H$, a unital $*$-representation $\pi:B\to B(H)$ and a unit vector $\xi\in H$ that satisfies
\[
\phi(b)=\langle \pi(b)\xi,\xi\rangle,\quad b\in B.
\]
When $\xi$ happens to be a cyclic vector for $\pi$, then the representation $(\pi,H,\xi)$ is unitarily equivalent to the GNS representation of $\phi$. In general, we will need to consider non-cyclic representations as well. 

Next, we define the subset $\D(S,B)\subset \E(B)\times \E(B)$ to consist of those pairs of states $(\phi,\psi)$ for which there are representations $(\pi,H,\xi)$ and $(\sigma,K,\eta)$ of $\phi$ and $\psi$ respectively, along with an isometry $V:H\to K$ satisfying $V\xi=\eta$, such that
\[
\pi(a)=V^*\sigma(a)V, \quad a\in S.
\]
Following \cite{DK2019} and \cite{DK2021}, we call  $\D(S,B)$ the \emph{dilation order relative to $S$ in $B$}. It is readily seen that $\D(S,B)$ is a preorder on $\E(B)$. 
It is in fact a partial order; this follows from \cite[Proposition 7.2.8 and Theorem 8.5.1]{DK2019}.

For the purposes of Arveson's hyperrigidity conjecture, the importance of the dilation order  is made manifest in the following fact, which follows from  \cite[Proposition 5.2.3 and Theorem 8.3.7]{DK2019}.

\begin{theorem}\label{T:Dmaxuep}
Let $B$ be a unital $\rC^*$-algebra and let $S\subset B$ be an operator system such that $\rC^*(S)=B$. Let $\pi:B\to B(H)$ be a unital $*$-representation with unit cyclic vector $\xi$. Then, the following statements are equivalent.
\begin{enumerate}[{\rm (i)}]
\item  $\pi$ has the unique extension property with respect to $S$.
\item The state 
$
b\mapsto \langle \pi(b)\xi,\xi\rangle
$
is $\D(S,B)$-maximal.
\end{enumerate}
\end{theorem} 

Using the terminology from Section \ref{S:boundaries}, we can now reformulate Arveson's conjecture.

\begin{corollary}\label{C:Arvreform}
Let $B$ be a separable unital $\rC^*$-algebra and let $S\subset B$ be an operator system such that $\rC^*(S)=B$. Assume that every pure state on $B$ is $\D(S,B)$-maximal. Then, the following statements are equivalent.
\begin{enumerate}[{\rm (i)}]
\item The operator system $S$ is hyperrigid in $B$.
\item The pre-order $\D(S,B)$ is hyperrigid.
\item There exists a $\D(S,B)$-boundary.
\end{enumerate}
\end{corollary}
\begin{proof}
(i)$\Rightarrow$(ii): By assumption, every state on $B$ is $\D(S,B)$-maximal, so trivially $\D(S,B)$ is hyperrigid.

(ii)$\Rightarrow$(i): By assumption, the set $\Omega$ of pure $\D(S,B)$-maximal states coincides with the set of pure states on $B$. Thus, $\Sigma_\Omega$ is the entire state space by   \cite[Theorem 4.2]{bishop1959}. Assuming that $\D(S,B)$ is hyperrigid, we thus conclude that every state on $B$ is $\D(S,B)$-maximal.

(ii)$\Leftrightarrow$(iii): This is contained in Corollary \ref{C:bdryHR}.
\end{proof}

\subsection{The maximal $\rC^*$-cover}

The dilation order does not depend only on the operator system, but also on the ambient $\rC^*$-algebra. For the ``maximal" representation of the operator system, the dilation order admits an alternative description, more in line with the classical Choquet order alluded to in the introduction and examined in \cite{DK2021}.  We recall some of the details underlying this non-trivial fact.

Operator systems can be defined abstractly with no mention of an ambient $\rC^*$-algebra or concrete representation on Hilbert space by means of the Choi--Effros theorem \cite[Theorem 13.1]{paulsen2002}. The following concept is useful in studying the flexibility afforded by this coordinate-free approach.

Let $S$ be an operator system.  A \emph{$\rC^*$-cover} of $S$ is a pair $(B,\theta)$ consisting of a unital $\rC^*$-algebra $B$ and a unital completely isometric map $\theta:S\to B$ such that $B=\rC^*(\theta(S))$.  It is known \cite{KW1998} that there exists a maximal $\rC^*$-cover $(A,j)$. More precisely, this $\rC^*$-cover has the property that, given any other $\rC^*$-cover $(B,\theta)$, there is a surjective unital $*$-homomorphism $\pi:A\to B$ such that $\pi\circ j=\theta$. It is customary to use the notation $A=\rC^*_{\max}(S)$ for this maximal $\rC^*$-cover. Further, it is well known that $\rC^*_{\max}(S)$ satisfies a formally stronger condition, which we record next for later reference.

\begin{lemma}\label{L:Cstarmax}
Let $S$ be an operator system and let $\phi:S\to B(H)$ be a unital completely contractive map. Then, there is a unital $*$-representation $\widehat\phi:\rC^*_{\max}(S)\to B(H)$ such that $\widehat\phi\circ j=\phi$.
\end{lemma}
\begin{proof}
Consider the map $j\oplus \phi:S\to \rC^*_{\max}(S)\oplus B(H)$, which is clearly unital and completely isometric. Define also the unital $*$-homomorphism $\pi_0:\rC^*_{\max}(S)\oplus B(H)\to B(H)$ as the projection onto the second component. Now, there is a unital $*$-homomorphism $\sigma:\rC^*_{\max}(S)\to \rC^*((j\oplus \phi)(S))$ such that $\sigma\circ j=j\oplus \phi$. It thus suffices to put $\widehat\phi=\pi_0\circ \sigma.$
\end{proof}

\subsection{Non-commutative functions}
Our next goal is to give an alternative description of the dilation order. Let us first set some notation regarding matrices indexed by sets with infinite cardinality.

 Let $H$ be a Hilbert space and let  $S\subset B(H)$ be an operator system. Let $m$ be a cardinal number. Let $H^{(m)}=\bigoplus_{n<m}H$, where the direct sum is taken over all cardinal numbers $n<m$. 
In standard fashion, a bounded linear operator $T$ on $H^{(m)}$ corresponds to a matrix $[t_{i,j}]_{i,j<m}$ with $t_{i,j}:H\to H$ such that 
\[
\sup\{ \|[t_{i,j}]_{i,j\in I}\|: I \text{ finite subset of } \{n<m\}\}
\]
is finite. Here, given a finite subset $I$ of $\{n<m\}$, we identify $[t_{i,j}]_{i,j\in I}$ with an operator on the Hilbert space $\bigoplus_{n\in I}H$. Accordingly, we may view $\bM_m(S)\subset B(H^{(m)})$ as those matrices with entries in $S$ for which the corresponding collection of finite submatrices is bounded. In particular, $\bM_m(S)$ is another operator system.

Let $\kappa$ be an infinite cardinal greater than the linear dimension of $S$.  Given a cardinal number $n\leq \kappa$, we fix a Hilbert space $H_{n}$ of dimension $n$.  We simply write $\bC$ for $H_1$ and $B(H_1)$. 

We let $K_n$ denote the set of all unital completely positive maps from $S$ into $B(H_n)$. Thus, $K_n$ is a convex subset of the unit ball of the space of all completely bounded maps from $S$ into $B(H_n)$. In addition, each $K_n$ compact in the topology of pointwise weak-$*$ convergence. 

The collection $K=(K_n)_{n\leq\kappa}$ enjoys some additional compatibility relations betweens each of its levels $K_n$, which makes it an \emph{nc convex set}; see \cite[Definition 2.2.1 and Example 2.2.6]{DK2019} for details.

Fix a cardinal $m\leq \kappa$. An \emph{nc function} $F:K\to \bM_m(\bC)$ is a collection  of functions
\[
F_n:K_n\to \bM_m(B(H_n)), \quad n\leq \kappa
\]
satisfying some natural compatibility and equivariance conditions; see \cite[Definition 4.2.1]{DK2019}. 
For each $n\leq \kappa$, we can write $F_n$ as a matrix $[f^{(i,j)}_{n}]_{i,j<m}$, where $f^{(i,j)}_{n}:K_n\to B(H_n)$. Correspondingly, each \[F^{(i,j)}=(f^{(i,j)}_{n})_{n\leq \kappa}:K\to \bC\] is an nc function.

If $F$ is self-adjoint, we say that it is  \emph{convex} if, for each cardinal $n\leq\kappa$, each pair of elements $\phi,\psi\in K_n$ and each number $0\leq t\leq 1$, we have
\[
F_n(t\phi+(1-t)\psi)\leq tF_n(\phi)+(1-t)F_n(\psi) \quad \text{ in } \bM_m(B(H_n)).
\]

An nc function $f:K\to\bC$ is said to be \emph{continuous} if each $f_n:K_n\to B(H_n)$ is continuous, where $B(H_n)$ is endowed with the $\sigma$-strong$^*$ topology \cite[Section II.2]{takesaki2002}. 
The set of all continuous nc functions from $K$ into $\bC$ is a unital $\rC^*$-algebra, which we will denote by $\A$.  By virtue of  \cite[Theorem 4.4.3]{DK2019}  there is a $*$-isomorphism
$
\Phi:\rC^*_{\max}(S)\to \A
$
such that
\begin{equation}\label{Eq:ncfunct}
[\Phi(b)](\phi)=\widehat\phi(b)
\end{equation}
for each element $b\in \rC^*_{\max}(S)$, each map $\phi\in K_n$ and each cardinal $n\leq \kappa.$
Therefore, any continuous nc function $F=(F^{(i,j)})_{i,j< m}: K\to \bM_m(\bC)$ gives rise to a matrix $[\Phi^{-1}(F^{(i,j)})]_{i,j< m}$ with entries in $\rC^*_{\max}(S)$.

\subsection{The cone $\Xi$}

We define $\Xi\subset \rC^*_{\max}(S)$ to be the set consisting of elements of the form 
\[\sum_{i,j\in I}c_j \ol{c_i}\Phi^{-1}(F^{(i,j)})\]
where $F=(F^{(i,j)})_{i,j<m}:K\to \bM_m(\bC)$ is a convex continuous nc function  for some cardinal $m\leq \kappa$, $I$ is some finite subset of $\{n<m\}$, and $\{c_i:i\in I\}$ is a set of complex numbers.

\begin{lemma}\label{L:convcone}
The set $\Xi$ is a cone in $\rC^*_{\max}(S)$. 
\end{lemma}
\begin{proof}
Let $m$ and $n$ be two cardinal numbers at most equal to $\kappa$. Let $F:K\to \bM_m(\bC)$ and $G:K\to \bM_n(\bC)$ be two convex continuous nc functions. Define a function $H:K\to \bM_{m+n}(\bC)$ as
\[
H(x)=F(x)\oplus G(x), \quad x\in K.
\]
It is easily verified that $H$ is still a convex continuous nc function. 

Next, let $I\subset \{r<m\}$ and $J\subset \{r<n\}$ be two finite subsets, and correspondingly let $\{c_i:i\in I\}$ and $\{d_j:j\in J\}$ be two finite subsets of complex numbers. Let $s,t\geq 0$. It is easily verified that there exists a finite subset $\Lambda\subset \{r<n+m\}$ and finitely many complex numbers $\{\alpha_\lambda :\lambda\in \Lambda\}$ such that
\begin{align*}
&s\left(\sum_{i,i'\in I}c_{i'} \ol{c_i}\Phi^{-1}(F^{(i,i')}) \right)+t\left(\sum_{j,j'\in J}d_{j'} \ol{d_j}\Phi^{-1}(G^{(j,j')}) \right)\\
&=\sum_{\lambda,\lambda'\in \Lambda} \alpha_{\lambda'}\ol{\alpha_\lambda} \Phi^{-1}(H^{(\lambda,\lambda')}).
\end{align*}
Hence, $\Xi$ is indeed a cone.
\end{proof}

Next, we wish to gain a more concrete understanding of the elements of $\Xi$. For this purpose, we will consider restrictions of nc functions on $K$ to the first level $K_1$, which is simply the state space of $S$. 

More precisely, consider the unital $*$-homomorphism $\rho:\A\to \rC(K_1)$ defined as
\[
[\rho(f)](\phi)=f_1(\phi)
\]
for each continuous nc function $f=(f_n):K\to \bC$ and each $\phi\in K_1$.
Define also the natural evaluation map $\eps:S\to \rC(K_1)$ as
\[
[\eps(s)](\phi)=\phi(s)
\]
for each $s\in S$ and $\phi\in K_1.$
This is a unital completely contractive map, so by Lemma \ref{L:Cstarmax}, there is a unique unital surjective $*$-homomorphism $q:\rC^*_{\max}(S)\to \rC(K_1)$ such that $q\circ j=\eps$ on $S$. For $s\in S$ and $\phi\in K_1$, using \eqref{Eq:ncfunct} we see that
\[
[(\rho\circ \Phi\circ j)(s)](\phi)=[\Phi(j(s))]_1(\phi)=\widehat\phi(j(s))=\phi(s)=[\eps(s)](\phi)
\]
so that $\rho\circ \Phi\circ j=\eps=q\circ j$ on $S$. This immediately implies that
\begin{equation}\label{Eq:restriction}
\rho\circ \Phi=q.
\end{equation}
We can now give a fairly concrete description of the restrictions to $K_1$ of the elements in $\Xi$.

\begin{proposition}\label{P:maxstable}
Let $S$ be a operator system with state space $L$. Let $\Gamma\subset \rC(L)$ denote the closed cone of continuous convex functions on $L$. Let $\eps:S\to \rC(L)$ be the evaluation map, and let $q:\rC^*_{\max}(S)\to \rC(L)$ denote the surjective unital $*$-homomorphism satisfying $q\circ j=\eps$ on $S$. Then, $q(\Xi)\subset \Gamma$ and $q(\Xi)$ contains all restrictions to $L$ of affine weak-$*$ continuous functions on $S^*$. Furthermore, $\Gamma$ is the smallest closed cone of $\rC(L)$ stable under maxima and  containing $q(\Xi)$.
\end{proposition}
\begin{proof}
Recall that for each cardinal $n\leq \kappa$, $K_n$ is the set of all completely positive maps from $S$ into $B(H_n)$. Hence, $K_1=L$. 

Fix $s\in S$. The evaluation function $\widehat s:K\to \bC$ defines a continuous nc function, which is readily seen to be convex. By \eqref{Eq:ncfunct} we see that $\Phi^{-1}(\widehat s)=j(s)$, so that $j(s)\in \Xi$. Hence $q(\Xi)$ contains $q(j(s))=\eps(s)$ for every $s\in S$, and hence contains all restrictions to $L$ of affine weak-$*$ continuous functions on $S^*$.

Next, let $\xi\in \Xi$. By definition, this means that there is a cardinal $m\leq \kappa$, a convex continuous nc function $F=(F^{(i,j)})_{i,j<m}: K\to \bM_m(\bC)$, a finite subset  $I\subset \{n<m\}$, and a subset $\{c_i:i\in I\}$ complex numbers such that
\[
\xi=\sum_{i,j\in I}c_j \ol{c_i}\Phi^{-1}(F^{(i,j)}).
\]
Let $h=(h_n)$ be the vector in $\bigoplus_{n<m} \bC$ such that $h_n=c_n$ if $n\in I$, and $h_n=0$ otherwise. Then, using that $F$ is convex and applying \eqref{Eq:restriction}, we obtain for each $\phi,\psi\in K_1$ and $0\leq t\leq 1$, that
\begin{align*}
[q(\xi)] (t\phi+(1-t)\psi)&=\sum_{i,j\in I}c_j \ol{c_i}F_1^{(i,j)}(t\phi+(1-t)\psi)\\
&=\langle F_1 (t\phi+(1-t)\psi)h,h\rangle_{H_m}\\
&\leq \langle (tF_1 (\phi)+(1-t)F_1(\psi))h,h\rangle_{H_m}\\
&=t\left(\sum_{i,j\in I}c_j \ol{c_i}F_1^{(i,j)}(\phi)\right)+(1-t)\left(\sum_{i,j\in I}c_j \ol{c_i}F_1^{(i,j)}(\psi)\right)\\
&=t[q(\xi)](\phi)+(1-t)[q(\xi)](\phi).
\end{align*}
We infer that $q(\xi)$ is convex on $L$.

We have thus proved that $q(\Xi)$ contains all restrictions to $L$ of affine weak-$*$ continuous functions on $S^*$, and it is contained in $\Gamma$. The second conclusion follows directly from this, in light of \cite[Corollary I.1.3]{alfsen1971}.
\end{proof}

\subsection{The pre-order induced by $\Xi$}

The motivation for introducing $\Xi$ is the next development. We define 
\[
\Order(\Xi)\subset \E(\rC^*_{\max}(S))\times \E(\rC^*_{\max}(S))
\]
to be the pre-order consisting of those pairs of states $(\phi,\psi)$ satisfying 
$
\phi(\xi)\leq \psi(\xi)$ for every $\xi\in \Xi.$

\begin{proposition}\label{P:dilordercone}
Let $S$ be an operator system with maximal $\rC^*$-cover  $(\rC^*_{\max}(S),j)$. Then, $\D(j(S),\rC^*_{\max}(S))=\Order(\Xi)$. In particular, $\D(j(S),\rC^*_{\max}(S))$ is a convex, weak-$*$ closed partial order on the state space of $\rC^*_{\max}(S)$.
\end{proposition}
\begin{proof}
By \cite[Theorem 8.5.1]{DK2019}, we see that $(\phi,\psi)\in \D(j(S),\rC^*_{\max}(S))$ is equivalent to 
\[
[\phi(\Phi^{-1}(F^{(i,j)}))]_{i,j<m}\leq [\psi(\Phi^{-1}(F^{(i,j)}))]_{i,j<m} \quad \text{ in } \bM_m(\bC)\]
for every convex nc continuous function $F=[F^{(i,i)}]_{i,j<m}:K\to \bM_m(\bC)$ and every cardinal number $m\leq \kappa$. For such a function $F$, the required inequality is equivalent to
\[
\langle [ \phi(\Phi^{-1}(F^{(i,j)}))]h,h\rangle\leq \langle [\psi(\Phi^{-1}(F^{(i,j)}))]h,h\rangle
\]
for every finitely supported vector $h$ in $\bigoplus_{n<m}\bC$. 
Given such  finitely supported vector $h$, there is a finite subset $I\subset \{n<m\}$ and complex numbers $\{c_i:i\in I\}$ such that
\begin{align*}
\langle [\phi(\Phi^{-1}(F^{(i,j)}))]h,h\rangle=\phi\left(\sum_{i,j\in I}c_j \ol{c_i}\Phi^{-1}(F^{(i,j)})\right)
\end{align*}
and 
\[
\langle [ \psi(\Phi^{-1}(F^{(i,j)}))]h,h\rangle=\psi\left(\sum_{i,j\in I}c_j \ol{c_i}\Phi^{-1}(F^{(i,j)})\right).
\]
Therefore, $(\phi,\psi)\in \D(j(S),\rC^*_{\max}(S))$  is equivalent to 
\[
\phi(\xi)\leq \psi(\xi), \quad \xi\in \Xi.
\]
In other words, $\D(j(S),\rC^*_{\max}(S))=\Order(\Xi)$.
It is routine to check that this implies that $\D(j(S),\rC^*_{\max}(S))$ is convex and weak-$*$ closed.
\end{proof}

The previous result applies only to operator systems that are represented in their maximal $\rC^*$-covers. Our next aim is to show that Proposition \ref{P:dilordercone} still contains relevant information about dilation maximal states for any representation of $S$. For this purpose, we need the following.

\begin{lemma}\label{L:maxinv}
Let $B$ be a unital $\rC^*$-algebra generated by an operator system $S\subset B$. Let $q:\rC^*_{\max}(S)\to B$ be the surjective unital $*$-homomorphism such that $q\circ j=\id$ on $S$. Let $\phi$ be a state on $B$. Then, $\phi$ is $\D(S,B)$-maximal if and only if $\phi\circ q$ is $\D(j(S),\rC^*_{\max}(S))$-maximal.
\end{lemma}
\begin{proof}
Let $\pi:B\to B(H)$ be the GNS representation of $\phi$. It is then easily verified that $\pi\circ q$ is the GNS representation of $\phi\circ q$. Since $q$ is completely isometric on $S$, we infer that $\pi$ has the unique extension property with respect to $S$ if and only if $\pi\circ q$ has the unique extension property with respect to $j(S)$; see for instance \cite[Theorem 2.1.2]{arveson1969}, the proof of which easily adapts outside the irreducible setting. The desired conclusion then follows from Theorem \ref{T:Dmaxuep}.
\end{proof}

Retaining the notation from above, we can now state the main result of this section.

\begin{theorem}\label{T:dilmaxinv}
Let $B$ be a unital $\rC^*$-algebra generated by an operator system $S\subset B$. Let $q:\rC^*_{\max}(S)\to B$ be the surjective unital $*$-homomorphism such that $q\circ j=\id$ on $S$. Then, $\max(\D(S,B))=\max(\Order(q(\Xi)))$.
\end{theorem}
\begin{proof}
Fix $\phi$ a state on $B$. Assume that $\phi$ is $\D(S,B)$-maximal. Let $\psi$ be another state on $B$ such that $(\phi,\psi)\in \Order(q(\Xi))$. Then, $(\phi\circ q,\psi\circ q)\in \Order(\Xi)$ and therefore $(\phi\circ q,\psi\circ q)\in \D(j(S),\rC^*_{\max}(S))$ by Proposition \ref{P:dilordercone}. On the other hand, $\phi\circ q$ is $\D(j(S),\rC^*_{\max}(S))$-maximal by Lemma \ref{L:maxinv}, whence $\phi\circ q=\psi\circ q$ and $\phi=\psi$. We conclude that $\phi$ is $\Order(q(\Xi))$-maximal. 

Conversely, assume that $\phi$ is $\Order(q(\Xi))$-maximal.  Let $\psi$ be another state on $B$ such that $(\phi,\psi)\in \D(S,B)$. Thus, there are representations $(\pi,H,\xi)$ and $(\sigma,K,\eta)$ of $\phi$ and $\psi$ respectively, along with an isometry $V:H\to K$ such that $V\xi=\eta$, satisfying
\[
\pi(a)=V^*\sigma(a)V, \quad a\in S.
\]
Since $(\pi\circ q,H,\xi)$ and $(\sigma\circ q, K,\eta)$ are representations of $\phi\circ q$ and $\psi\circ q$ respectively, it easily follows that $(\phi\circ q,\psi\circ q)\in \D(j(S),\rC^*_{\max}(S))$. Hence, another application of Proposition \ref{P:dilordercone} gives $(\phi\circ q,\psi\circ q)\in \Order(\Xi)$, so that $(\phi,\psi)\in \Order(q(\Xi))$. The maximality of $\phi$ in $\Order(q(\Xi))$ then forces $\phi=\psi$, so that $\phi$ is $\D(S,B)$-maximal.
\end{proof}

Let us explore some of the ramifications of the previous result in relation to Arveson's conjecture in the commutative setting. 

Let $B$ be a separable commutative unital $\rC^*$-algebra and let $S\subset B$ be an operator system with $\rC^*(S)=B$. Let $L$ denote the state space of $S$. Because $B$ is commutative, the evaluation map $\eps:S\to \rC(L)$ is  completely isometric. 
Let $q:\rC^*_{\max}(S)\to \rC(L)$ denote the unique surjective unital $*$-homomorphism such that $q\circ j=\eps$ on $S$. 

Under the assumption that all pure states on $\rC(L)$ are $\D(\eps(S),\rC(L))$-maximal, to establish the conjecture, we need to prove that $\D(\eps(S),\rC(L))$ is hyperrigid (see Corollary \ref{C:Arvreform}). On the other hand,  the property of being hyperrigid only depends on the set of maximal elements,  so we may replace the dilation order by any pre-order with the same maximal elements. In light of Theorem \ref{T:dilmaxinv}, this means that can just as well try to show that $\Order(q(\Xi))$ is hyperrigid.

In turn, by Corollary \ref{C:bdryHR}, this is equivalent to the existence of a boundary for $\Order(q(\Xi))$. In this context, we may thus hope to apply the classical machinery of \cite[Corollary I.5.18]{alfsen1971} to construct such a boundary.  This strategy essentially reduces to the one employed in \cite{DK2021}. Indeed, in order for  \cite[Corollary I.5.18]{alfsen1971} to be applicable, the cone $q(\Xi)$ would need to be stable under taking maxima. If this were the case, then by virtue of Proposition \ref{P:maxstable}, we would know that the closure of $q(\Xi)$ coincides with the cone of all continuous convex functions on $L$. In turn, Theorem \ref{T:dilmaxinv} would then imply that the dilation maximal elements coincide with so-called Choquet maximal elements, which are at the heart of \cite{DK2021}.
\subsection{Uniqueness of $\Xi$}

It is natural now to wonder whether $\Xi$ is the unique cone in $\rC^*_{\max}(S)$ that satisfies Proposition \ref{P:dilordercone}. Before we can address this question, we introduce some notation and terminology.

Let $B$ be a unital $\rC^*$-algebra. Given a pre-order $\Delta\subset \E(B)\times \E(B)$, we define the \emph{induced cone} of $\Delta$ to be the set $\Cone(\Delta)\subset B$ of self-adjoint elements $b$ with the property that $\phi(b)\leq \psi(b)$ whenever $(\phi,\psi)\in \Delta$.  If, conversely, we are given a cone $\Gamma\subset B$ of self-adjoint elements, we define the \emph{induced order} of $\Gamma$ to be the pre-order $\Order(\Gamma)\subset \E(B)\times \E(B)$ consisting of those pairs of states $(\phi,\psi)$ satisfying 
\[
\phi(\gamma)\leq \psi(\gamma), \quad \gamma\in \Gamma.
\]
There exists a certain duality between these objects, as we show next.

\begin{theorem}\label{T:duality}
Let $B$ be a unital $\rC^*$-algebra and let $\Gamma\subset B$ be a cone of  self-adjoint elements containing both $1$ and $-1$. Then, $\Cone(\Order(\Gamma))$ is the norm closure of  $\Gamma$. 
\end{theorem}
\begin{proof}
By continuity, it follows from the definitions that the norm closure of $\Gamma$ is contained in $ \Cone(\Order(\Gamma))$. Assume that there is a self-adjoint element $b\in \Cone(\Order(\Gamma))$ outside the norm closure of $\Gamma$. By the convex separation theorem, we can find a bounded linear functional $\theta$ on $B$ such that
\[
 \sup_{c\in \Gamma}(\re \theta)(c)<(\re \theta)(b).
\]
Here, we let $\re \theta=(\theta+\theta^*)/2$; this is a self-adjoint bounded linear functional on $B$.

Next, let $c\in \Gamma$. Then $tc\in \Gamma$ for every $t>0$, so that
\[
t(\re \theta)(c)<(\re \theta)(b),
\]
which forces $(\re \theta)(c)\leq 0.$
Since both $1$ and $-1$ belong to $\Gamma$, we find $(\re \theta)(1)=0$. Therefore,
\begin{equation}\label{Eq:reneg}
\sup_{c\in \Gamma}(\re \theta)(c)=0.
\end{equation}
Next, by the Jordan decomposition \cite[Lemma 3.2.2]{pedersen1979book}, there are positive linear functionals $\phi_0,\psi_0$ on $B$ with the property that $\re\theta=\phi_0-\psi_0$. Using that $(\re \theta)(1)=0$, we infer that there is strictly positive number $r$ such that  $r=\|\phi_0\|=\phi_0(1)=\psi_0(1)=\|\psi_0\|$. Define $\phi=\frac{1}{r}\phi_0$ and $\psi=\frac{1}{r}\psi_0$, which are then states on $B$ satisfying
\[
 \sup_{c\in \Gamma} (\phi(c)-\psi(c)) =0<\phi(b)-\psi(b)
\]
by virtue of \eqref{Eq:reneg}. In particular, $(\phi,\psi)\in \Order(\Gamma)$. In turn, because $b$ lies in $ \Cone(\Order(\Gamma))$, this means that $\phi(b)\leq \psi(b)$, contradicting the previous inequality.
\end{proof}

One may wonder if the ``duality" uncovered above between cones and pre-orders goes in the other direction, namely whether $\Order(\Cone(\Delta))=\Delta$ for any pre-order $\Delta$ on the state space of $B$. It is readily seen that $\Delta$ must be convex and weak-$*$ closed for this to hold, but at the time of this writing we do not know if these necessary conditions are also sufficient.

We can now address the uniqueness question raised earlier.

\begin{corollary}\label{C:dilconeunique}
Let $S$ be an operator system. Then, the norm closure of $\Xi$ is the unique closed cone $\Gamma$ of self-adjoint elements in $\rC^*_{\max}(S)$ containing $1$ and $-1$, and satisfying \[\D(j(S),\rC^*_{\max}(S))=\Order(\Gamma).\]
\end{corollary}
\begin{proof}
First, it is clear that $\Order(\Xi)=\Order(\ol{\Xi})$, where $\ol{\Xi}\subset B$ denotes the norm closure of $\Xi$. Hence, we can apply Proposition \ref{P:dilordercone} to see that 
\[
\D(j(S),\rC^*_{\max}(S))=\Order(\ol{\Xi}).\]
Assume that $\Gamma\subset \rC^*_{\max}(S)$ is a closed cone of self-adjoint elements, containing $1$ and $-1$, and satisfying \[\D(j(S),\rC^*_{\max}(S))=\Order(\Gamma).\] By virtue of Theorem \ref{T:duality}, we see that 
$
\ol{\Xi}=\Cone(\D(j(S),\rC^*_{\max}(S)))=\Gamma.
$
\end{proof}

\section{Detecting hyperrigidity with the boundary projection}\label{S:bdryproj}

In the discussion following Theorem \ref{T:dilmaxinv}, we saw that the existence of a boundary for the dilation order cannot simply be inferred from the known classical techniques of \cite{alfsen1971}, even in the commutative setting. In this section, we exploit  non-commutative machinery to exhibit a certain ``non-classical" boundary. We also illustrate how its regularity properties are deeply intertwined with Arveson's conjecture. The crucial idea is to consider absolute continuity for states, and how it interacts with maximality in the dilation order.

Let $B$ be a unital $\rC^*$-algebra. The bidual $B^{**}$ is then a von Neumann algebra. If $\pi:B\to B(H)$ is a  $*$-representation, then it admits a unique weak-$*$ continuous extension $\tilde\pi: B^{**}\to B(H)$, which is also a $*$-representation. 
Given a state $\phi$ on $B$, we consider its left kernel
\[
L_\phi=\{x\in B^{**}:\phi(x^*x)=0\}.
\]
Let $\psi$ be another state on $B$. We say that $\phi$ is \emph{absolutely continuous} with respect to $\psi$ if $L_\psi\subset L_\phi$. 

As shown in \cite[Lemma 2.6]{CT2023}, this definition coincides with the usual one when $B$ is commutative. Unlike in the classical setting however, the existence of some form of a Radon--Nikodym theorem in the general case is a rather subtle issue, and no perfect analogue exists as far as we know; see  \cite{sakai1965},\cite{PT1973},\cite{exel1990},\cite{vaes2001},\cite{GK2009} and the references therein. Fortunately, this difficulty can be circumvented via the following fact.

\begin{lemma}\label{L:abscont}
Let $B$ be a unital $\rC^*$-algebra. Let $\phi,\psi$ be states on $B$ with respective GNS representations $(\pi,H,\xi)$ and $(\sigma,K,\eta)$. Assume that $\phi$ is absolutely continuous with respect to $\psi$. Then, the following statements hold.
\begin{enumerate}[{\rm (i)}]
\item There is a unique weak-$*$ continuous $*$-representation $\rho:\tilde\sigma(B^{**})\to \tilde \pi(B^{**})$ such that $\rho\circ \tilde \sigma=\tilde \pi$.
\item There is a normal state $\tau$ on $B(K)$ such that $\phi=\tau\circ \tilde\sigma$.
\end{enumerate}
\end{lemma}
\begin{proof}
(i)  The vector $\xi$ is cyclic for $\pi$, so that an element $x\in B^{**}$ belongs to $\ker \tilde\pi$ if and only if 
\[
 \phi((xb)^*(xb))=\|\tilde\pi(x)\pi(b)\xi\|^2=0
\]
for every $b\in B$. In other words, 
\[
\ker  \tilde\pi=\{x\in B^{**}:xB\subset L_\phi\}.
\]
Similarly, 
\[
\ker  \tilde\sigma=\{x\in B^{**}:xB\subset L_\psi\}.
\]
Now, $\phi$ is absolutely continuous with respect to $\psi$, so $L_\psi\subset L_\phi$. We thus infer that $\ker \tilde\sigma\subset \ker \tilde \pi$. 

Next, there are central projection $\fs_\pi, \fs_\sigma\in B^{**}$ satisfying
\[
\ker \tilde\pi=B^{**}(I-\fs_\pi) \qand \ker \tilde\sigma=B^{**}(I-\fs_\sigma).
\]
By the previous paragraph, we see that $\fs_\pi\leq \fs_\sigma$. Hence, the weak-$*$ continuous $*$-homomorphism  $\rho_0:  B^{**}\fs_\sigma\to B^{**} \fs_\pi$ of multiplication  by $\fs_\pi$ is surjective. Furthermore, there are weak-$*$ homemorphic $*$-isomorphisms $\theta_\pi:\tilde\pi(B^{**})\to B^{**} \fs_\pi$ and $\theta_\sigma:\tilde\sigma(B^{**})\to B^{**} \fs_\sigma$ defined as
\[
\theta_\pi(\tilde\pi(x))=x\fs_\pi \qand \theta_\sigma(\tilde\sigma(x))=x\fs_\sigma
\]
for every $x\in B^{**}$. The desired map is then $\rho=\theta_\pi^{-1}\circ \rho_0 \circ \theta_\sigma$, and it is clearly unique.

(ii) Define a weak-$*$ continuous state $\tau_0:B(H) \to \bC$ as 
\[
\tau_0(t)=\langle t\xi,\xi  \rangle, \quad t\in B(H).
\]
By \cite[Proposition 1.24.5]{sakai1971}, there is a weak-$*$ continuous state $\tau$ on $B(K)$ extending $\tau_0\circ \rho$, and this state has the desired properties.
\end{proof}

This fact will be exploited in the following fashion. 

\begin{proposition}\label{P:abscontconv}
Let $B$ be a unital $\rC^*$-algebra. Let $\phi,\psi$ be states on $B$ such that  $\phi$ is absolutely continuous with respect to $\psi$.  Let $\sigma:B\to B(K)$ be the GNS representation of $\psi$. Then, the following statements hold.

\begin{enumerate}[{\rm (i)}]
\item There is a countable orthonormal set of vectors $\{e_n\}$ in $K$ along with  positive numbers $\{t_n\}$ such that $\sum_{n=1}^\infty t_n=1$ and
\[
\phi(b)=\sum_{n=1}^\infty t_n \langle \sigma(b)e_n,e_n \rangle, \quad b\in B.
\]
\item The GNS representation of $\phi$ is unitarily equivalent to a subrepresentation of $\sigma^{(\infty)}$.
\end{enumerate}
\end{proposition}
\begin{proof}
(i)  By Lemma \ref{L:abscont}, there is a normal state $\tau$ on $B(K)$ such that $\phi=\tau\circ \tilde\sigma$. Hence, there is a positive trace class operator $T$ on $K$ with $\tr(T)=1$ such that
\[
\phi(b)=\tr(\sigma(b)T), \quad b\in B.
\]
Applying the spectral theorem to $T$, there is a countable orthonormal set of vectors $\{e_n\}$ in $K$ along with  positive numbers $\{t_n\}$ such that $\sum_{n=1}^\infty t_n=1$ and
\[
\phi(b)=\sum_{n=1}^\infty t_n \langle \sigma(b)e_n,e_n \rangle, \quad b\in B.
\]

(ii) We see that $\xi=(t_n^{1/2}e_n)$ is a unit vector in $K^{(\infty)}$ and satisfies
\begin{align*}
\phi(b)&=\langle \sigma^{(\infty)}(b)\xi,\xi\rangle, \quad b\in B.
\end{align*}
It follows that the GNS representation of $\phi$ is unitarily equivalent to the restriction of $\sigma^{(\infty)}$ to the cyclic subspace generated by $\xi$.
\end{proof}

There are two noteworthy consequences of this proposition. First, we state a complement to the ``approximate" Radon--Nikodym theorem found in \cite[Theorem 2.7]{CT2023} that may be of independent interest.

\begin{corollary}\label{C:RN}
Let $B$ be a unital $\rC^*$-algebra. Let $\phi,\psi$ be states on $B$ such that  $\phi$ is absolutely continuous with respect to $\psi$. Then, $\phi$ belongs to the norm closure of the convex hull of
\[
\{\psi(b^* \cdot b):b\in B, \psi(b^*b)=1\}.
\]
\end{corollary}
\begin{proof}
Let $(\sigma,K,\eta)$ be the GNS representation of $\psi$. Let $\xi\in K$ be a unit vector. Given $\eps>0$, cyclicity of $\eta$ implies that there is $b\in B$ such that $\psi(b^*b)=1$ and $\|\sigma(b)\eta-\xi\|<\eps$. In turns, this implies that
\[
\|\langle \sigma(\cdot )\xi,\xi\rangle -\psi(b^*\cdot b)\|<2\eps.
\]
Hence, the state $\langle \sigma(\cdot)\xi,\xi\rangle$ lies in the norm-closure of $\{\psi(b^* \cdot b):b\in B, \psi(b^*b)=1\}.$ Proposition \ref{P:abscontconv}(i) then yields the desired conclusion.
\end{proof}

The second consequence of Proposition \ref{P:abscontconv} is the crucial technical step in achieving our goal in this section.

\begin{theorem}\label{T:abscont}
Let $B$ be a unital $\rC^*$-algebra and let $S\subset B$ be an operator system such that $B=\rC^*(S)$.  Let $\phi,\psi$ be states on $B$ such that  $\phi$ is absolutely continuous with respect to $\psi$.  If $\psi$ is $\D(S,B)$-maximal, then so is $\phi$.
\end{theorem}
\begin{proof}
Let $\sigma:B\to B(K)$ be the GNS representation of $\psi$. Since $\psi$ is $\D(S,B)$-maximal, we infer from Theorem \ref{T:Dmaxuep} that $\sigma$ has the unique extension property with respect to $S$. In turn, by virtue of  \cite[Lemma 2.8]{CTh2022}, we see that any subrepresentation of $\sigma^{(\infty)}$ also has the unique extension property with respect to $S$. In particular, this is the case for the GNS representtion of $\phi$  by Proposition \ref{P:abscontconv}(ii). Another application of Theorem \ref{T:Dmaxuep} implies that $\phi$ is also $\D(S,B)$ maximal.
\end{proof}

We can prove the main result of the paper.

\begin{theorem}\label{T:bdryproj}
Let $B$ be a unital $\rC^*$-algebra and let $S\subset B$ be an operator system such that $B=\rC^*(S)$. Then, the following statements hold.
\begin{enumerate}[{\rm (i)}]
\item The set of $\D(S,B)$-maximal elements  is a norm-closed face of $\E(B).$

\item There exists a projection $\fd \subset B^{**}$ with the property that a state $\phi$ on $B$ is $\D(S,B)$-maximal precisely when $\phi(\fd)=1$.
\end{enumerate}
\end{theorem}
\begin{proof}
(i) Let $\theta$ be a $\D(S,B)$-maximal state on $B$. Assume that there is $0<t<1$ and states $\phi,\psi$ on $B$ such that $\theta=t\phi+(1-t)\psi$. It is easily verified that both $\phi$ and $\psi$ are absolutely continuous with respect to $\theta$, so that $\phi,\psi$ are also $\D(S,B)$-maximal by Theorem \ref{T:abscont}. This shows that the $\D(S,B)$-maximal elements form a face of the state space of $B$.

Next, let $(\phi_n)$ be a sequence of $\D(S,B)$-maximal states converging in norm to some state $\phi$ on $B$. For each $n$, let $\sigma_n:B\to B(H_n)$ be the GNS representation of $\phi_n$, which has the unique extension property with respect to $S$ by Theorem \ref{T:Dmaxuep}. Arguing as in the proof of Proposition \ref{P:abscontconv}(ii), we see that the GNS representation of $\sum_{n=1}^\infty \frac{1}{2^n}\phi_n$ is a subrepresentation of $\bigoplus_{n=1}^\infty \sigma_n$, and hence also has the unique extension property with respect to $S$ \cite[Lemma 2.8]{CTh2022}. Therefore, $\sum_{n=1}^\infty \frac{1}{2^n}\phi_n$ is $\D(S,B)$-maximal by Theorem \ref{T:Dmaxuep}. A  routine calculation reveals that  $\phi$ is absolutely continuous with respect to $\sum_{n=1}^\infty \frac{1}{2^n}\phi_n$, whence $\phi$ is $\D(S,B)$-maximal by Theorem \ref{T:abscont}. 

(ii) By (i), we may apply  \cite[Theorem 3.5]{CT2023} to find a projection $\fd\in B^{**}$ with the property that a state $\phi$ on $B$ is absolutely continuous with respect to some $\D(S,B)$-maximal state precisely when 
\[
\phi(b)=\phi(\fd b), \quad b\in B.
\]
A standard multiplicative domain argument reveals that this is simply equivalent to $\phi(\fd)=1$. Another application of Theorem \ref{T:abscont} completes the proof.
\end{proof}

The projection $\fd$ in the previous result completely determines the set of $\D(S,B)$-maximal elements: these are the states that are ``concentrated" on $\fd$. This phenomenon is reminiscent of the notion of a $\D(S,B)$-boundary from Section \ref{S:boundaries}. For this reason, we call $\fd$ the \emph{boundary projection} of the dilation order.

\subsection{Hyperrigidity and non-commutative topology}

We saw in Corollary \ref{C:Arvreform} that the existence of a $\D(S,B)$-boundary would yield a positive solution to Arveson's conjecture.
For the remainder of the paper, we examine whether there is a similar relationship between the boundary projection $\fd$ and the conjecture.

By virtue of Theorem \ref{T:dilmaxinv}, we know that there is a weak-$*$ closed convex pre-order on $\E(B)$ with the same maximal elements as $\D(S,B)$. If $B$ is separable, the set $\Omega$ of pure $\D(S,B)$-maximal states is therefore Borel measurable by Lemma \ref{L:Borel}. Recall that we denote by $\Sigma_\Omega$ those states $\phi$ on $B$ for which there is a Borel probability measure concentrated on $\Omega$ such that $\phi=\int_\Omega \omega d\mu(\omega)$.

A projection $q\in B^{**}$ is said to be \emph{closed} if it is the decreasing weak-$*$ limit of a net of contractions in $B$ \cite{akemann1969},\cite{akemann1970left}. When $B$ is separable, the net can be chosen to be a sequence. The projection $q$ is \emph{open} if $I-q$ is closed.

We now give a generalization of  \cite[Theorem 3.2]{bishop1959}. 

\begin{corollary}\label{C:BdL}
Let $B$ be a separable unital $\rC^*$-algebra and let $S\subset B$ be an operator system such that $B=\rC^*(S)$. Any state $\phi\in \Sigma_\Omega$ satisfies $\phi(q)=0$ if $q\in B^{**}$ is a closed projection orthogonal to the boundary projection $\fd$.
\end{corollary}
\begin{proof}
Fix a closed projection $q\in B^{**}$ orthogonal to $\fd$.
Let $\phi$ be a state in $\Sigma_\Omega$. We can find a Borel probability measure $\mu$ concentrated on the set $\Omega$ such that
\[
\phi(b)=\int_{\Omega} \omega(b)d\mu(\omega), \quad b\in B.
\]
Next, note that $\omega(q)\leq \omega(I-\fd)=0$ for every $\omega\in \Omega$ by Theorem \ref{T:bdryproj}. Since $q$ is closed and $B$ is separable, there is a decreasing contractive sequence $(b_n)$ converging to $q$ in the weak-$*$ topology of $B^{**}$. By the dominated convergence theorem, we find
\[
\phi(q)=\lim_{n\to\infty}\int_\Omega\omega(b_n)d\mu(\omega)=\int_\Omega \omega(q)d\mu(\omega)=0
\]
as desired.
\end{proof}

We make some remarks about the previous result. 

First, recall from Corollary \ref{C:Arvreform} that we are interested in determining if the dilation order is hyperrigid. Using the boundary projection from Theorem \ref{T:bdryproj}, this condition is simply saying that $\phi(\fd)=1$ whenever $\phi\in \Sigma_\Omega$. Corollary \ref{C:BdL} yields a similar, albeit weaker, statement. Roughly speaking, it can be thought of as saying that each state $\phi\in \Sigma_\Omega$ is ``concentrated" on the boundary projection $\fd$. Indeed, the conclusion is a non-commutative analogue of a standard way of formalizing the idea that a measure is concentrated on some possibly non-measurable set; see \cite[Theorem 3.4]{bishop1959},\cite[Section 9]{DK2019}.  

Second, assume that $\Delta\subset \E(B)\times \E(B)$ is a weak-$*$ closed, convex pre-order whose pure maximal elements are all $\D(S,B)$-maximal as well. We may then apply Theorem \ref{T:BdL} to see that any $\Delta$-maximal state lies in $\Sigma_\Omega$, so that they all satisfy the conclusion of Corollary \ref{C:BdL}. In particular, this principle can be applied when $B$ is commutative and $\Delta$ is the classical Choquet order \cite{DK2021}.

Finally, as an application of Corollary \ref{C:BdL}, we show how hyperrigidity can be reformulated in non-commutative topological terms using the boundary projection.

\begin{corollary}\label{C:HRtop}
Let $B$ be a separable unital $\rC^*$-algebra and let $S\subset B$ be an operator system such that $B=\rC^*(S)$. Assume that every pure state on $B$ is $\D(S,B)$-maximal. 
Then, the following statements are equivalent.
\begin{enumerate}[{\rm (i)}]
\item The operator system $S$ is hyperrigid in $B$.
\item The boundary projection $\fd$ is closed.
\item The boundary projection $\fd$ is the infimum of a collection of open projections in $B^{**}$.
\end{enumerate}
\end{corollary}
\begin{proof}
(i) $\Rightarrow$ (ii): In this case, all states on $B$ are $\D(S,B)$-maximal, so that $\fd=I$ by Theorem \ref{T:bdryproj}, which is then trivially closed.

(ii) $\Rightarrow$ (iii): This follows from \cite[Proposition 2.3]{hay2007}.

(iii) $\Rightarrow$ (i): Assume that there is a collection $\U\subset B^{**}$ of open projections such that $\fd=\wedge\{u:u\in \U\}$. Fix $u\in \U$. Then, $I-u$ is a closed projection orthogonal to $\fd$. Since every pure state is $\D(S,B)$-maximal, every state on $B$ lies in $\Sigma_\Omega$ by \cite[Theorem 4.2]{bishop1959}. Hence, applying Corollary \ref{C:BdL}, we see that $\phi(I-u)=0$ for every state $\phi$ on $B$, and consequently $u=I$. It follows that $\fd=I$.
\end{proof}

\bibliography{orders}
\bibliographystyle{plain}

\end{document}